\documentclass[a4paper]{amsart}
\usepackage[utf8]{inputenc}
\usepackage{amssymb,pstricks,amscd}
\usepackage[T1]{fontenc}
\usepackage[totalwidth=15cm,totalheight=22cm]{geometry}
\usepackage{graphicx}
\usepackage[title,titletoc,toc]{appendix}

\newtheorem{cor}{Corollary}[section]
\newtheorem{thm}[cor]{Theorem}
\newtheorem{prop}[cor]{Proposition}
\newtheorem{lemma}[cor]{Lemma}

\theoremstyle{definition}
\newtheorem{defi}[cor]{Definition}
\theoremstyle{remark}
\newtheorem{remark}[cor]{Remark}

\newcommand{\N}{{\mathbb N}}
\newcommand{\Q}{{\mathbb Q}}

\newcommand{\R}{{\mathbb R}}

\newcommand{\be}{\begin{eqnarray}}
\newcommand{\ee}{\end{eqnarray}}

\renewcommand{\S}{{\mathbb S}^m}
\newcommand{\Sp}{\mathbb S}
\def\Ri{\mathbb{R}^+ \cup \{+\infty\}}
\renewcommand{\P}{\mathcal{P} (\S)}
\def\G{\Gamma (\sigma,\mu)}

\newcommand{\C}{\mathcal{C}}
\newcommand{\F}{\mathcal{F}}

\renewcommand{\vec}[1]{\overrightarrow{#1}}
\renewcommand{\tilde}[1]{\widetilde{#1}}
\renewcommand{\bar}[1]{\overline{#1}}
\newcommand{\ep}{\varepsilon}
\newcommand{\supp}{\mbox{supp}}

\renewenvironment{keywords}{
       \list{}{\advance\topsep by0.35cm\relax\small
       \leftmargin=1cm
       \labelwidth=0.35cm
       \listparindent=0.35cm
       \itemindent\listparindent
       \rightmargin\leftmargin}\item[\hskip\labelsep
                                     \bfseries Keywords:]}
     {\endlist}

\def\res{\mathop{\hbox{\vrule height 7pt width .5pt depth 0pt \vrule height .5pt width 6pt depth 0pt}}\nolimits}

\begin{document}

\title{Prescription of Gauss curvature using Optimal Mass Transport}
\author{J\'er\^ome Bertrand}
\address{Institut de Math\'ematiques de Toulouse, UMR CNRS 5219 \\
Universit\'e Toulouse III \\
31062 Toulouse cedex 9, France}
\email{jerome.bertrand@math.univ-toulouse.fr}

\begin{abstract}
In this paper we give a new proof of a theorem by Alexandrov on the Gauss curvature prescription of Euclidean convex sets. This proof is based on the duality theory of convex sets and on optimal mass transport. A noteworthy property of this proof is that it does not rely neither on the theory of convex polyhedra nor on P.D.E. methods (which appeared in all the previous proofs of this result). 
\end{abstract}

\maketitle


\section{Introduction}
In this paper, we consider the problem of prescribing the Gauss curvature (in a generalised measure-theoretic sense) of convex sets in the Euclidean space. Our aim is  to prove this geometrical statement by solving an appropriate variational problem. 

\subsection{Alexandrov's problem in Euclidean space}
We start with the description of the Euclidean result which was proved by A.D. Alexandrov in  \cite{Alex42, AlexCP}. 

In order to state the result, we recall the notion of Gauss curvature measure introduced by Alexandrov.  Consider a convex body ({\it i.e.} a closed bounded convex set whose interior is nonempty) $\Omega$ in $\R^{m+1}$ and assume that the origin of $\R^{m+1}$ is located within $\Omega$. Under these assumptions, the map
\begin{equation}\label{Ie1}\begin{array}{rccc} \vec{\rho} : &\S &\longrightarrow &\partial \Omega \\
                                          &x &\longmapsto      &\rho(x)x
 \end{array}
 \end{equation}
 is a homeomorphism  (where the radial function $\rho$ is defined by $\rho(x) = \sup \{ s; sx \in \Omega\}$). Note that $\rho$ is a Lipschitz function bounded away from $0$ and $\infty$.

 The Gauss curvature measure is the Borel probability measure 
$$ \mu := \sigma ({\mathcal G}\circ \vec{\rho}(\cdot))$$
where $\sigma$ stands for the uniform Borel probability measure on $\S$ (considered as the unit sphere centered at the origin) and ${\mathcal G} : \partial \Omega \rightrightarrows \S$ is the Gauss multivalued map. In other terms, the Gauss curvature measure is the {\it pull-back} of the uniform measure through the map ${\mathcal G}\circ \vec{\rho}$. The curvature measure $\mu$ is well-defined since
\begin{equation}\label{me32}
\sigma\left(\left\{n \in \S; \,\exists \,x_1\neq x_2\in \S ; \,{\mathcal G}(\rho(x_1)x_1) ={\mathcal G}(\rho(x_2)x_2)=n\right\}\right)=0
\end{equation}
(see \cite[Lemma 5.2]{bak} for a proof). 

Note also that the curvature measure depends on the location of the origin within the convex body and is invariant under homotheties about that point. 

By using the assumption $0 \in \,\stackrel{\circ}{\Omega}$, it is easy to convince oneself that for all non-empty spherical convex set $\omega \subsetneq \S$, 
 
\begin{equation}\label{assump}
 \mu(\omega) < \sigma (\omega_{\pi/2})
 \end{equation}
where $\omega_{\pi/2}= \{x \in \S; d(x,\omega)<\pi/2\}$ and $d(\cdot,\cdot)$ is the standard distance on $\S$.
 
\begin{remark}\label{remtriv} Note that (\ref{assump}) implies that $\mu$ cannot be supported in a closed hemisphere. Consequently, the distance between an arbitrary point in $\S$ and the support of $\mu$ is smaller than $\frac{\pi}{2}$.
\end{remark}
Alexandrov's theorem states that (\ref{assump}) is actually a sufficient condition for $\mu$ arising from this construction.
%
 \begin{thm}[Alexandrov]\label{Imain1} Let $\sigma$ be the uniform probability measure on $\S$ and $\mu$ be a Borel probability measure on $\S$ satisfying for any non-empty convex set $\omega \subsetneq \S$,

\[ \mu(\omega) < \sigma (\omega_{\pi/2}).\]

 Then, there exists a unique convex body  in $\R^{m+1}$ containing $0$ in its interior (up to homotheties) whose $\mu$ is the Gauss curvature measure.
 \end{thm}

 Our proof of Alexandrov's result is an easy consequence of a well-known result in optimal mass transport (but for non-standard cost function thus the proof is non trivial): Theorem \ref{transp}. Alexandrov's theorem is deduced from the latter result in Section \ref{esubE} through an elementary argument based on the classical notion of polar transform of bounded sets. 
 
\subsection{Optimal mass transport}

In this part, we introduce notation and briefly describe the optimal mass transport problem on $\S$. For much more on the subject, we refer to \cite{Vi03,Vi08}. This problem involves two probability measures, denoted by $\mu,\sigma \in \P$ in the sequel, and a cost function $c: \S\times \S \rightarrow \Ri $. We also need to introduce the set of {\emph transport plans} $\Gamma(\sigma,\mu)$, namely the set of probability measures $\Pi \in \mathcal{P}(\S\times \S)$ such that for any Borel set $A\subset \S$
\begin{equation}\label{planmargi} \sigma(A) = \Pi(A\times \S) \mbox{  and  } \mu(A)= \Pi(\S\times A).
\end{equation}
The transport plans can also be characterized in terms of continuous functions as follows. Given $f: \S \rightarrow \R$ a continuous fonction, it holds
\begin{equation}\label{planconti}
\int_{\mathbb{S}^m \times  \mathbb{S}^m } f(n) \, d\Pi (n,x)= \int_{\S} f(n) \, d\sigma(n) \mbox{ and } \int_{\mathbb{S}^m \times  \mathbb{S}^m } f(x) \, d\Pi (n,x)= \int_{\S} f(x) \, d\mu(x).
\end{equation}
Equipped with the topology induced by the weak convergence of probability measures, the set $\G$ is a {\emph compact set} as a consequence of the Banach-Alaoglu theorem. 

The cost function $c$ we consider is defined by the formula
\begin{equation}\label{costdefi}
 c(n,x) = \left\{\begin{array}{lr}
-\log \langle n, x \rangle = -\log \cos d(n,x) & \mbox{ if } d(n,x) < \pi/2 \\
+ \infty & \mbox{ otherwise }
\end{array}\right.
\end{equation}

The cost function $c$ satisfies a standard set of assumptions in the field (highlighted in the lemma below) with the important exception that it is not real-valued. Therefore, some standard results do not apply to $c$. 
\begin{lemma} The cost function $c: \S \times \S \longrightarrow \Ri $ defined in (\ref{costdefi}) is a continuous map. Moreover, restricted to the open set $\{c < + \infty\}$, the function $c$ is a strictly convex and increasing smooth function of the spherical distance. Consequently, for $(n,x)$ in any fixed open set $\Omega$ such that  
$\bar{\Omega} \subset \{c < + \infty\}$, the function $(n,x) \rightarrow c(n,x)$ is a Lipschitz differentiable function.
\end{lemma} 
\begin{remark} The set $\Ri$ is endowed with the order topology.
\end{remark}

The mass transport problem consists in studying

\begin{equation}\label{massdefi}
 \inf_{\Pi \in \G} \int_{\S \times \S} c(n,x)\, d\Pi(n,x).
 \end{equation}

Note that the compactness of $\G$ combined with the (lower semi-)continuity of $c$ yields existence of minimizers in the problem above. These minimizers are called \emph{optimal} (transport) plans. To study the property of the optimal plans  (including the question of uniqueness), Kantorovitch introduced a \emph{dual problem}. Let us define ${\mathcal A}$ as the set of pairs $(\phi,\psi)$ of Lipschitz functions defined on $\S$ that satisfy $\phi(n)+\psi(x) \leq c(n,x)$ for all $x,n \in \S$. Kantorovitch's variational problem consists in studying
\begin{equation}\label{massdefi2}
\sup_{(\phi,\psi) \in \mathcal{A}}  \left\{ \int_{\S} \phi(n) d\sigma(n) + \int_{\S} \psi(x) d\mu(x)\right\}. 
\end{equation}

It is easy to see that the quantity above is always smaller or equal to (\ref{massdefi}). Indeed, given $(\phi,\psi) \in \mathcal{A}$ and $\Pi \in \G$, Property (\ref{planconti}) allows us to write
$$ \int_{\S} \phi(n) \,d\sigma(n)+ \int_{\S}\psi(x) \, d\mu(x)= \int_{\S\times \S} (\phi(n)+\psi(x)) \,d\Pi(n,x) \leq \int_{\S\times \S} c(n,x) \, d\Pi(n,x).$$

It can be proved that $ (\ref{massdefi}) =(\ref{massdefi2})$ whenever the cost function is continuous and nonnegative; this type of result is called Kantorovitch's duality. However, since the cost function assume infinite values, the common value could be infinite. Besides, for some (real-valued) Lipschitz cost functions and when the base space is compact, existence and uniquess of the solution of Kantorovitch's variational problem can be proved. On the contrary, there are cases (involving non real-valued cost function) where solution of Kantorovitch's problem does not exist. We refer to \cite{Vi08} for more on this.

The main result of the paper is a proof of a strong form of the Kantorovitch duality relative to the non real-valued cost function $c$.
\begin{thm}[Strong Kantorovich duality, \cite{Oliker}]\label{transp}
Consider $\mu$ and $\sigma$ as in Alexandrov's theorem. Then, the following equality holds.
\begin{equation}\label{KD5}
\max_{(\phi,\psi) \in\mathcal{A}}  \left\{ \int_{\S} \phi(n) d\sigma(n) + \int_{\S} \psi(x) d\mu(x)\right\} =
\min_{\Pi \in \Gamma (\sigma, \mu)} \int_{\S\times \S} c(n,x) \,d\,\Pi(n,x) <  +\infty.
\end{equation}
Moreover, Kantorovitch's variational problem admits a unique solution $(\phi,\psi)$ (up to replacing $(\phi,\psi)$ by $(\phi+s,\psi-s)$ with $s\in \R$). If $(\phi,\psi) \in \mathcal{A}$ is a solution of the variational problem then, 
 \begin{eqnarray}\label{concave}
  \psi (x) = \min_{n \in \S} c(n,x) - \phi(n) \nonumber \\
  \phi (n) = \min_{x \in \S} c(n,x) - \psi(x) 
  \end{eqnarray}
 for all $n \in \S$ and $\mu$-a.e. $x \in \S$. 
\end{thm}

The proof of Theorem \ref{transp} is in two steps. First, we show that the quantities in (\ref{KD5}) are finite. More precisely we show a result which, we believe, is of independent interest (in connection with Theorem \ref{BMCC} for instance).

\begin{thm}\label{finite} Let $\mu$ be a Borel probability measure on $\S$.  Then, $\mu$ satisfies Alexandrov's condition (\ref{assump}) if and only if there exists $\Pi \in \Gamma(\sigma,\mu)$ such that 
$$  c \in L^{\infty} (\Pi).$$
\end{thm}

Then, we show the existence of maximizers of Kantorovitch's variational problem as well as their uniqueness up to constants. Note that as a by-product of our proof, we get an analogue of the Brenier-McCann theorem for the cost $c$ stated below.

\begin{thm}\label{BMCC} Let $f\sigma$ and $\mu$ two probability measures on $\S$ such that there exists $\Pi \in \Gamma(f\sigma,\mu)$ for which $c\in L^{\infty}(\Pi)$. Then, the mass transport problem
\begin{equation}\label{min}
\min_{\Pi \in \Gamma (f\sigma, \mu)} \int_{\S\times \S} c(n,x) \,d\,\Pi(n,x) <  +\infty
\end{equation}
admits a unique solution $\Pi_0$. Moreover, $\Pi_0= (Id,T)_{\sharp} f\sigma$ where for $\sigma$-a.e. $n \in \S$,
$$ T(n) = \mbox{\rm exp}_n \left( \frac{-\arctan |\nabla \phi(n)|}{|\nabla \phi(n)|}\nabla \phi(n) \right)$$
being $\phi$ a Lipschitz $c$-concave function.
\end{thm}
 We refer to Definition \ref{cost} for the definition of $c$-concave function. We don't know whether the assumption that $c\in L^{\infty}(\Pi)$ for a well-chosen transport plan $\Pi$ can be weakened to (\ref{min}). The details of the proof are left to the reader (however a proof is sketched in Remark \ref{rem}).

 \subsection{Comments, related results, and organization of the paper.}\label{comment}
  
  Our motivation to study this problem comes from a paper by Oliker \cite{Oliker}. In this paper, Oliker proves Alexandrov's theorem through the study of Kantorovich's variational problem. Moreover his proof, as Alexandrov's one, consists in establishing the result for convex polyhedra first and then "passing to the limit". This requires some fine estimates (see the proofs of \cite[Theorems 6 and 7]{Oliker}). This paper led us to the question of whether it was possible to deduce Alexandrov's theorem from optimal mass transport techniques without using convex polyhedra. Roughly speaking, the arguments relying on the theory of convex polyhedra are replaced in our proof by elementary compactness results relative to the Hausdorff metric. Moreover, taken for granted that Kantorovitch's variational problem can be solved, optimal mass transport gives a very simple argument to prove that the given measure is indeed the Gauss curvature measure of the underlying convex body. We also mention that the paper is self-contained (up to elementary properties of Hausdorff metric and a lemma on c-cyclical monotonicity whose proof is less than a page) and does not require any specific knowledge in optimal mass transport.

Under the assumption that the Gauss curvature measure is absolutely continuous, the regularity of the convex set has been studied by Pogorelov in two dimensions \cite{pogorelov} and by Oliker \cite{oliker2} in higher dimensions. Let us also mention a paper by Treibergs \cite{treibergs} where the author proves {\it a priori} bounds for the ratio of circumscribed and inscribed radii of the convex body depending on the curvature measure. A study of other curvature measures in the smooth case has been carried out  in \cite{guan} (see also the references therein). 
  
To conclude, let us also mention that the same approach has been applied in \cite{bert} where we prove an analogue of Alexandrov's theorem for compact hyperbolic orbifolds. However in the hyperbolic setting the relative cost function is real-valued and standard results in optimal mass transport can be applied. Adapting known results in optimal mass transport to the non-standard cost function involved in Alexandrov's problem is the main achievement of this paper. Building on the results of this paper, we consider in a joint work in progress with Castillon, the problem of prescribing the Gauss curvature of convex bodies in the hyperbolic space. 

In Section 2, we infer Theorem \ref{Imain1} from Theorem \ref{transp}. In Section 3, we use the Hausdorff metric to establish two reinforcements of Assumption (\ref{assump}). In the last section, we prove Theorem \ref{finite} and Theorem \ref{transp}.


  \vspace*{0.2cm}
  

 \section{Proof of Alexandrov's theorem}\label{esubE}
 
 In this part we prove Alexandrov's theorem by means of Theorem \ref{transp}. We start with some preliminary remarks on support functions. Let us recall that the support function $h$ of a convex body $\Omega$ is defined by the formula
 
\begin{equation}\label{support}
h(n) = \sup_{x \in \S} \left\{\rho(x) \langle x,n\rangle  \right\}
\end{equation}

Note that $h(n)= \rho(x) \langle x,n\rangle $ if and only if the hyperplane orthogonal to $n$ through $\stackrel{\rightarrow}{\rho}(x)$ supports the convex $\Omega$. In other words, this equality amounts to  
\begin{equation}\label{me23}
n \in {\mathcal G}\circ \vec{\rho}(x).
\end{equation}

More generally, we can consider the radial and support functions of star-shaped sets with respect to $0$. More precisely, we denote by $ {\mathcal E} $ the set of non-empty subsets of $ \R^{m+1}$ that are closed, bounded and star-shaped with respect to $0$ and whose radial function is positive and continuous.

It is worth mentioning that the support function $h_F$ of any set $F \in {\mathcal E}$ is related to the radial function of its polar set $F^{\circ}$ in the following way:
\begin{equation}\label{truc2}
 \frac{1}{h_F}  = \rho_{F^{\circ}}
\end{equation} 
where $F^{\circ}= \{n \in \R^{m+1}; \,\forall x \in F\; \langle x,n\rangle \leq 1 \} $.

Let us recall that $F^{\circ}$ is a convex set belonging to ${\mathcal E}$. Moreover by definition of a polar set, it is always true that $F^{\circ \circ} \supset F$ and the equality holds if and only $F \in {\mathcal E}$ is a convex set (we refer to \cite{schneider} for more on polar sets, including a proof of the previous equivalence, see Theorem 1.6.1). In terms of radial functions, this amounts to
\begin{equation}\label{truc}
\rho_{F^{\circ \circ}} \geq \rho_F
\end{equation}
and (using that $F$ is star-shaped) that the equality holds if and only if $F \in {\mathcal E}$ is convex.

To summarize, a set $\C$ is a convex body with $0$ in its interior if and only if its support and radial functions are related by (\ref{support}) and 

\begin{equation}\label{radial} 
\rho(x) =\rho_{\C^{\circ \circ}}(x)=  \frac{1}{h_{\C^{\circ}}}(x) =\frac{1}{\sup_{n \in \S}\frac{\langle x,n\rangle}{h(n)}}=\inf_{n \in \S; \langle x,n\rangle >0} \left\{\frac{h(n)}{\langle x,n\rangle}\right\}  
\end{equation}

Now, if we set $\phi = \ln (1/h)$ and $\psi = \ln(\rho)$, these functions are well-defined Lipschitz functions. Moreover (\ref{support}) and (\ref{radial}) can be rephrased as 

 \begin{eqnarray*}
  \psi (x) = \min_{n \in \S} c(n,x) - \phi(n) \nonumber \\
  \phi (n) = \min_{x \in \S} c(n,x) - \psi(x) 
  \end{eqnarray*}
 for all $x,n \in \S$ and $c$ is the cost function defined in (\ref{costdefi}). 
 
(Note that to get the formula above, we use the fact that the extremums are realized by $x,n$ such that $\langle n,x \rangle >0$.) In other terms, $\phi$ and $\psi$ are c-concave functions according to the theory of optimal mass transport. Conversely, the datum of two Lipschitz c-concave functions $\phi,\psi$ obviously determines a unique convex body.

To complete the proof of the existence part of Alexandrov's theorem, it remains to show that, given a solution $(\phi,\psi)$ of Kantorovitch's variational problem, the Gauss curvature measure of the convex body determined by  $(\phi,\psi)$ is indeed $\mu$. To this aim, we select $\Pi_0$ an optimal transport plan and notice that the equality (\ref{KD5}) reads
$$ \int_{\S} \phi(n) d\sigma(n) + \int_{\S} \psi(x) d\mu(x)= \int_{\S\times \S} c(n,x) \,d\,\Pi_0(n,x) $$
which, using (\ref{planconti}), can be reformulated as
$$ \int_{\S\times \S} (c(n,x)-\phi(n)-\psi(x))\, d\Pi_0(n,x)=0.$$
Equivalently,
\begin{equation}\label{me60}
 \Pi_0( \{(n,x) \in (\S)^2; \phi(n) + \psi(x)= c(n,x) \})=1
 \end{equation}
since $\phi(n) + \psi(x) \leq c(n,x)$ for all $n,x \in \S$. Now, the discussion above allows us to rewrite (\ref{me60}) as
$$ \Pi_0( \{(n,x) \in (\S)^2; n \in \mathcal{G}(\vec{\rho}(x)) \})=1.$$

As a consequence, for any Borel set $U \subset \S$, it holds
\begin{eqnarray*}
\mu(U) &=& \Pi_0 (\S\times U)\\
             &=& \Pi_0(\S \times U\cap \{(n,x) \in (\S)^2; \phi(n) + \psi(x)= c(n,x) \})\\
             &=& \Pi_0(\S \times U\cap \{(n,x) \in (\S)^2; n \in \mathcal{G} (\vec{\rho}(x))\})\\
             &=& \Pi_0( \mathcal{G} \circ \vec{\rho} (U) \times U \cap  \{(n,x); n \in \mathcal{G}(\vec{\rho}(x))\})\\
             &=& \Pi_0( \mathcal{G}\circ \vec{\rho}(U) \times \S  \cap   \{(n,x); n \in  \mathcal{G}(\vec{\rho}(x))\})\\
            &=&  \Pi_0( \mathcal{G}\circ \vec{\rho}(U) \times \S ) \\
             &=& \sigma ( \mathcal{G}\circ \vec{\rho}(U))
\end{eqnarray*}

where we used in line 5 
\begin{multline*}
    \mathcal{G}\circ \vec{\rho}(U) \times  U^c\, \cap  \{(n,x) \in (\S)^2; n \in \mathcal{G}(\vec{\rho}(x)))\} \subset 
 \{(n,x)  \in (\S)^2 ; \exists\, x'\neq x, n \in   \mathcal{G}(\vec{\rho}(x))\cap   \mathcal{G}(\vec{\rho}(x'))\}
 \end{multline*}
 which yields
 \begin{multline*}
 \Pi_0(  \mathcal{G}\circ \vec{\rho}(U) \times  U^c \cap  \{(n,x) \in (\S)^2; n \in \mathcal{G}(\vec{\rho}(x)))\}) \leq \\
 \sigma( \{n  \in \S ; \exists\, x'\neq x, n \in   \mathcal{G}(\vec{\rho}(x))\cap   \mathcal{G}(\vec{\rho}(x'))\})=0
 \end{multline*}
according to (\ref{me32}). 

It remains to prove the uniqueness part. If two distinct convex bodies $\C_1,\C_2$  have the same curvature measure then their radial and support functions give rise to two distinct solutions of Kantorovitch's variational problem through the transformation $\phi_i = \ln (1/h_i)$ and $\psi_i= \ln (\rho_i)$ ($i \in \{1,2\}$). Therefore,  
Theorem \ref{transp} yields that $\phi_1-\phi_2= -(\psi_1-\psi_2)$ is constant which means that $\C_1$ and $\C_2$ are dilations of each other about the origin.

\section{Hausdorff convergence and self-improvement of Alexandrov's assumption (\ref{assump})}

This part is devoted to technical results that will be used in the subsequent parts to prove Theorem \ref{transp}.

In this section we recall some standard compactness results on the space of closed (or convex) sets of a given compact metric space. We restrict our attention to the case where the metric space is the round sphere. The topology on the closed sets is induced by the Hausdorff metric whose definition is recalled below. Then, we prove some continuity results on functionals defined on these spaces. Finally, we use these results to prove reinforcements of Alexandrov's assumption (\ref{assump}).

\subsection{Continuity and compactness results for the Hausdorff topology}

\begin{defi}[Hausdorff metric]Let $\F$ (respectively $\mathcal{C}$) be the set of {\emph nonempty} compact sets (respectively compact convex sets) in $\S$. For $V_1,V_2 \in \F$, the Hausdorff distance between these sets is defined by the formula
$$d_H (V_1,V_2) = \min \left\{ \epsilon; V_1 \subset \bar{(V_2)_{\epsilon}} \mbox{ and } V_2 \subset \bar{(V_1)_{\epsilon}} \right\}$$
where $V_{\epsilon} = \{y \in \S; d(y,V) < \epsilon\}$.
\end{defi}

\begin{prop} \label{compact}The set $\F$ endowed with $d_H$ is a compact metric space. Moreover, $\C$ is a closed subset of $\F$ for this topology.
\end{prop}
For more on the Hausdorff metric, including a proof of  the proposition above, we refer to Rockafellar's book \cite{Rocka70}. From now on, we assume that $\F$ and $\C$ are endowed with the Hausdorff metric. Note that a set $\omega \subset \S$ is convex if and only if the Euclidean cone generated by $\omega$ is convex. Consequently, the Hahn-Banach theorem yields that any $\omega \in \mathcal{C}\setminus \{\S\}$ is contained in a closed ball of radius $\pi/2$, and $\S$ is an isolated point in $\mathcal{C}$.

Let us recall that $\omega^*$, the polar set of $\omega$ on $\S$, is defined by the formula
\begin{equation}\label{polar} \omega^* = \{ y \in \S; \langle y,x\rangle \leq 0 \,\forall x \in \omega\}= \{ y \in \S; d(y,\omega) \geq \pi/2\}.
\end{equation}
\begin{prop}\label{conti}The polar map 
$$
\begin{array}{rcl}
\C & \longrightarrow & \C \\
 \omega &\longmapsto &\omega^*
 \end{array}
$$ is continuous.
\end{prop}
\begin{remark}
The polar map is actually a local isometry. A proof of this fact can be found in \cite[p 501]{RockaWets}.
\end{remark}
 The next result is based on the following fact whose proof is an easy consequence of standard theorems in Measure theory.

\noindent  Fact: Let $\theta$ be a Borel probability measure on $\S$. The map 
 \begin{equation}\label{fact}
\begin{array}{rcl}
\F &\longrightarrow & [0,1]  \\
 V &  \longmapsto & \theta (V)
\end{array}
\end{equation}
 is upper semi-continuous.

 The above fact can be sharpened when we focus on $\C$. 
 
 \begin{lemma}\label{conticonv}
 The map
$$ \begin{array}{rcl}
\C &\longrightarrow & [0,1]  \\
 \omega &  \longmapsto & \theta (\omega)
\end{array}
$$
is continuous at any point where $\theta(\partial \omega)=0$.
\end{lemma}

\begin{proof}
According to Fact (\ref{fact}), it suffices to prove the lower semicontinuity. We can further assume that $\stackrel{\circ}{\omega} \neq \emptyset$ since otherwise $\theta(\omega)=0$ and the continuity follows from Fact (\ref{fact}). Fix a point $p \in \stackrel{\circ}{\omega}$. Let $P$ be the convex set obtained as the intersection of the cone generated by $\omega$ in $\R^{m+1}$ with the tangent space to $\S$ at $p$. We call  $\lambda P$ the Euclidean convex body obtained by dilating $P$ to a factor $\lambda$ with respect to $p$, and  $\lambda \omega$ the trace of $\lambda P$ on $\S$. Observe that $\stackrel{\circ}{\omega} \neq \emptyset$ implies $d(\partial \omega, \lambda\omega)>0$  when $0\leq \lambda <1$. Let $(\omega_n)_{n \in \N}$ be a sequence in $\C$ converging to $\omega$. Given $\ep>0$, we have $d_H (\omega, \omega_n) <\ep$ for large $n$. In particular, this implies $\partial \omega \subset \bar{(\omega_n)_{\ep}}$. Thus for large $n$, the convexity of $\omega_n$ yields 
$$ \omega_n \supset \lambda(\ep) \omega$$
where $ 0\leq 1- \lambda(\ep)$ is small (depending on $\ep$). Since $\theta(\partial \omega)=0$, $\lim_{\lambda \uparrow 1} \theta(\lambda \omega) = \theta(\omega)$ and the proof is complete.
\end{proof}
For later use, let us state the following result. 
 \begin{lemma} \label{conti2}The map  $\Theta$ defined below is lower semi-continuous on $\F \times [0,\frac{\pi}{2})$.
 
  $$ \begin{array}{rcl}
   \Theta: \F \times [0,\pi/2)& \longrightarrow & \R \\
           (V,\epsilon) &\longmapsto &\sigma (V_{\pi/2-\epsilon})                      
         \end{array}
         $$

 \end{lemma}
 
 \begin{proof} 
 Let $(\epsilon_n)_{n \in \N}$ be a sequence of positive numbers going to $\epsilon \in [0,\frac{\pi}{2})$ and $(V_n)_{n \in \N} \in \F^{\N}$ a sequence converging to $V_{\infty}$. We fix $k$ a positive integer. For large $n$, we have
 \begin{equation}\label{e20}
   (V_n)_{\pi/2- \epsilon_n}  \supset (V_n)_{\pi/2-(\epsilon+ \frac{1}{k})}.
 \end{equation}
 Up to enlarging $n$, we can also assume  that $V_{\infty} \subset (V_n)_{\frac{1}{k}}$. 
 This gives,
 $$  (V_{\infty})_{\pi/2-(\epsilon+ \frac{2}{k})}  \subset  (V_n)_{\pi/2-(\epsilon+ \frac{1}{k})} \subset  (V_n)_{\pi/2- \epsilon_n} 
 $$ 
 Since $ (V_{\infty})_{\pi/2-(\epsilon+ \frac{2}{k})}$ is an nondecreasing sequence of sets with respect to $k$ whose limit is $(V_{\infty})_{\pi/2- \epsilon}$, we have $\sigma((V_{\infty})_{\pi/2-(\epsilon+ \frac{2}{k})})) \rightarrow \sigma((V_{\infty})_{\pi/2- \epsilon})$. Since $k$ is arbitrary in our previous computations, this gives the result.
 
 \end{proof}
 
 \subsection{Reinforcement of Assumption (\ref{assump})} 
 In this part, we use the compactness of $\C$ and $\F$ to get some safety margin in (\ref{assump}). Intuitively speaking, this corresponds to the fact that the origin of $\R^n$ is located within the convex body we are looking for. In order to state our first improvement, let us notice that Alexandrov's assumption (\ref{assump}) can be reformulated as
 $$\mu(\S \setminus \omega) > \sigma (\omega^ *)$$
 where $w^*$ is defined in (\ref{polar}). The first statement is a straightforward consequence of the compactness of $\C$, Proposition \ref{conti}, (\ref{fact}), and Lemma \ref{conticonv}.
 
\begin{prop}\label{prop42} Let $\mu$  be a Borel probability measure satisfying  (\ref{assump}). There exists a positive number $\epsilon$ such that for all  $\omega \in \C\setminus \{\S\}$,

\begin{equation}\label{rein} 
\mu(\S \setminus \omega) > \sigma (\omega^ *) + \epsilon.
\end{equation}

\end{prop}
\begin{remark} By definition of $\C$, $\emptyset \not\in \C$.
\end{remark}
Now, we want to use the compactness of $\F$ in order to strengthen (\ref{assump}) in another way. To this aim, we note that for an arbitrary set $U$,
 \begin{equation}\label{convstar}
 U^* = (Conv(U))^*
 \end{equation} 
 where $U^*= \{ y \in \S; \langle y,x\rangle \leq 0 \,\forall x \in U\}$ and $Conv$ stands for the convex hull. Thus, (\ref{assump}) also holds true for all $V \in \F$ such that $Conv(V) \neq \S$. 

Assumption (\ref{assump}) can be improved as follows.

\begin{prop}\label{rein3} Let $\mu$ be a Borel probability measure satisfying (\ref{assump}). There exists $\alpha >0$ such that for all $F \in \F$,

\begin{equation}\label{rein2} 
\mu (F) \leq \sigma (F_{\pi/2-\alpha}).
\end{equation}
\end{prop}

\begin{remark} It was already known that the curvature measure of a convex body satisfies (\ref{rein2}) (of course, at that stage we don't know yet that $\mu$ is the curvature measure of such a set). This is the content of \cite[Theorem 1.B]{treibergs}. In the same paper, Treibergs also proves that the ratio circumscribed radius/inscribed radius of a convex body can be uniformly bounded from above by a universal function of $\alpha$ and $m$. On the contrary, he proves this result is false if we use (\ref{rein}) instead of (\ref{rein2}). 
\end{remark}

\begin{proof} 
We set $f_{\alpha} (F) = \Theta(F,\alpha)-\mu(F)$. The inequality in (\ref{rein2}) can be rephrased as $f_{\alpha} (F) \geq 0$. We start by investigating  the function $f_0$. Let $F \in \F$ be such that $Conv (F) \neq \S$. Then, Alexandrov's assumption (\ref{assump}) together with (\ref{convstar}) yield
$$f_0(F) = \sigma (\S \setminus F^*) - \mu(F) \geq  \sigma (\S \setminus (Conv(F)^*) -\mu(Conv(F)) >0.$$
To complement this result, let us prove that for $F \in \F$ such that $Conv(F)= \S$,  the inequality $f_0(F) \geq 0$ holds. Indeed, by assumption, there is no closed ball of radius $\pi/2$ containing $F$. This yields $\max_{x \in \S} d(x,F)< \pi/2$. In particular, we have  $F_{\pi/2}\supset \S$ which clearly entails $f_0(F) =1 -\mu(F) \geq 0.$\par

Conversely, if $F \in \F$ is such that $f_0(F)=0$ then, by what precedes, $Conv(F)=\S$. We then have $0=f_0(F)= 1 - \mu(F)$. In other terms,
\begin{equation}\label{detail12}
 F \supset \supp (\mu).
 \end{equation}

Note that Remark \ref{remtriv} yields that $\supp (\mu)$ is not contained in any closed ball of radius $\pi/2$ thus $Conv(\supp (\mu))=\S$. Therefore, as mentioned above $\max_{x \in \S} d(x,\supp(\mu))< \pi/2$.  For sufficiently small $\alpha>0$ (i.e. such that $\pi/2 - \alpha > \max_{x \in \S} d(x,\supp(\mu))\}$), we get $\supp(\mu)_{\pi/2-\alpha} \supset \S$ which in return implies
\begin{equation}\label{erein3}
 f_{\alpha}(\supp (\mu)) =0.
 \end{equation}

Now we can prove (\ref{rein2}) by contradiction. 
Assume that for any $\alpha >0$, there exists $F_{\alpha}$ such that $f_{\alpha}(F_{\alpha}) <0$. Note that $\mu(F \cap \supp (\mu))= \mu(F)$ implies $f_{\alpha}(F \cap \supp (\mu)) \leq f_{\alpha} (F)$, therefore  we can further assume that $F_{\alpha} \subset \supp (\mu)$. 

Since the compact subsets of $\supp (\mu)$ form a compact set relative to the Hausdorff metric, we obtain the existence of a converging sequence $F_{\alpha_n}$ to $F_0 \in \F$ where $\lim_{n \rightarrow 0} \alpha_n=0$, $F_0 \subset \supp (\mu)$, and
\begin{equation}\label{detaildetail}  f_{\alpha_n} (F_{\alpha_n}) <0.
\end{equation}
The lower-semicontinuity of $\Theta$ then gives $f_0(F_0) \leq 0$. Recall that $f_0$ is nonnegative. Consequently, we infer $f_0(F_0)= 0$ and, using (\ref{detail12}), $F_0 \supset \supp (\mu)$. On the other hand, we also know that  $F_0 \subset \supp (\mu)$.

To summarize, we have found a converging sequence $F_{\alpha_n}\subset \supp (\mu)$ to $\supp (\mu)$ such that (\ref{detaildetail}) holds. By definition of Hausdorff convergence, for any small $\beta>0$ and any sufficiently large $n$ (depending on $\beta$), we have $\bar{(F_{\alpha_n})_{\beta}} \supset \supp (\mu)$. Then, we can estimate
\begin{eqnarray*}
f_{\alpha_n}(F_{\alpha_n}) &=& \sigma ((F_{\alpha_n})_{\pi/2-\alpha_n}) -\mu(F_{\alpha_n})\\
								&\geq & \sigma((\bar{(F_{\alpha_n})_{\beta}})_{\pi/2-(\alpha_n+\beta)})- \mu(\supp(\mu))\\
								& \geq & \sigma(\supp(\mu))_{\pi/2-(\alpha_n+\beta)})- \mu(\supp(\mu) )\\
								&\geq & f_{\alpha_n +\beta}(\supp(\mu)). 
\end{eqnarray*}
Observe that (\ref{erein3}) implies $f_{\alpha_n +\beta}(\supp(\mu))=0$ for sufficiently small $\alpha_n,\beta$, hence a contradiction.

\end{proof}

 \section{Proof of Theorem \ref{transp}}
 \subsection{Well-posedness of the optimal transport problem}\label{WP}
In this part, we prove Theorem \ref{finite} whose straightforward consequence is
\begin{equation}\label{TP}
 \min_{\Pi \in \G} \int_{\S \times \S} c(n,x)\, d\Pi(n,x) <+\infty.
 \end{equation}
 Let us recall the statement of Theorem \ref{finite}.
\begin{thm}Let $ \mu$ be a Borel probability measure on $\S$ satisfying (\ref{assump}). There exists a plan $\Pi \in \Gamma(\sigma, \mu) $ such that
$$ c \in L^{\infty}(\Pi).$$
\end{thm}

\begin{remark}\label{rem42} The converse of the above result is straightforward. Indeed, if $c \in L^{\infty}(\Pi)$ then there exists $\alpha>0$ such that for $\Pi$-a.e. pair $(n,x)$, $d(n,x)\leq \pi/2-\alpha$. By definition of a transport plan, this implies that for any Borel set $U\subset \S$, $\mu(U) \leq \sigma (\bar{U_{\pi/2-\alpha}})$; in particular $\mu$ satisfies (\ref{assump}).
\end{remark}

\begin{proof}
Let $\mu$ be a probability measure on $\S$ satisfying (\ref{assump}). Thanks to Proposition \ref{rein3}, there exists a number $\alpha >0$ such that (\ref{rein2}) holds for any $F \in \F$. The first step of the proof is to show that we can approximate $\mu$ by a finitely supported measure $\tilde{\mu}$ that still satisfies (\ref{rein2}) (up to sligthly decreasing $\alpha$). To this aim, we first approximate $\mu$ by $(\mu * \rho_{\ep})_{\ep < \alpha/4}$, being $\rho_{\ep}$ a family of standard radial mollifiers on $\Sp^m$. We fix such a $\ep$ and set $ \hat{\mu}=\mu * \rho_{\ep}$;  by definition, $\hat{\mu}$ satisfies (\ref{rein2}) with $\frac{3\alpha}{4}$ instead of $\alpha$. Now, we claim there exists a finite partition $(U_i)_{i \in \{1, \cdots N\}}$ of $\S$ made of Borel sets such that  
\begin{equation}\label{blorp}
diam (U_i) < \ep \mbox{ and }   \hat{\mu}(U_i) \in \Q,
\end{equation}
a proof is given in the appendix.

For each $U_i$, choose $x_i \in U_i$ and set 
$$ \mu_e= \sum_{i=1}^{N} \hat{\mu}(U_i)\delta_{x_i}.$$
By assumption on the diameter of $U_i$, $\mu_e$ satisfies   
 \begin{equation}\label{me33} 
\forall F \in \F; \;\;\;\; \mu_e (F) \leq \sigma \left(\cup_{x \in F} B\left(x, \pi/2-\frac{\alpha}{2}\right)\right)
\end{equation}
and the proof of the first step is complete. According to (\ref{blorp}), $\mu_e$ can be rewritten (up to repeating some the $x_i$)
$$ \mu_e= \frac{1}{M} \sum_{i=1}^M \delta_{x_i}$$
with $M \in \N$.

The next step is to show the existence of $\Pi_e \in \Gamma (\sigma, \mu_e)$ such that
 $$ \int_{\S \times \S} c(n,x) \,d\Pi_e(n,x) \leq -\ln \left(\sin \left(\frac{\alpha}{4}\right)\right).$$
Up to enlarging $M$ and repeating $x_i$ if necessary, we can build by induction on the dimension $m$ another partition $(V_i)_{i=1}^{M}$ of $\S$ such that $diam( V_i) \leq \frac{\alpha}{4}$ and $\sigma(V_i)= \frac{1}{M}$ (see the appendix). We claim that the set-valued map 
 $$
 \begin{array}{rccc}
 F: &\{1,\cdots,M\} &\longrightarrow &\left\{V_s, s\in \{1, \cdots, M\}\right\} \\
       &        i & \longmapsto & \{V_s ; V_s \subset B(x_i, \pi/2- \frac{\alpha}{4})\}
  \end{array}
  $$
 satisfies the assumptions of the Marriage lemma. Indeed, consider $I$ a subset of $\{1,\cdots,M\}$. Thanks to (\ref{me33}), we have
 $$ \frac{\sharp I }{M} \leq \mu_e (\{x_i, i \in I\}) \leq \sigma \Big(\cup_{i \in I}B\Big(x_i, \pi/2-\frac{\alpha}{2}\Big) \Big).$$
Now, by assumption on the $V_i$'s, we get
 $$ \bigcup_{i \in I}  B\Big(x_i, \pi/2-\frac{\alpha}{2}\Big)  \subset \bigcup_{  V_s \in F(I)} V_s.$$  
 Combining these properties together proves that the assumptions are satisfied. Consequently, there exists a one-to-one map $f:   \{1,\cdots,M\} \longrightarrow \{V_i ; i \in \{1, \cdots ,M\}\}$ such that for all $i$, $f(i) \subset B(x_i, \pi/2- \frac{\alpha}{4})$. This fact clearly entails that the plan which maps the mass located at $x_i$ uniformly on $f(i)$ is a plan $\Pi_e$ in $\Gamma(\sigma,\mu_e)$ such that
\begin{equation}\label{me35} 
\Pi_e \left(\left\{(n,x) \in (\S)^2; d(n,x) \leq \pi/2- \frac{\alpha}{4}\right\}\right)=1.
\end{equation}
Note that the bound does not depend on $M$ nor on $\ep$. Therefore, by letting $\ep$ go to $0$, we can construct by the same method a sequence of empirical measures which converges to $\mu$, all of whose elements satisfy (\ref{me35}). Then using the Banach-Alaoglu theorem, we can extract a subsequence of plans which converges to an element of $\Gamma(\sigma,\mu)$ that satisfies (\ref{me35}).
\end{proof}


\subsection{Proof of Theorem \ref{transp}: the existence part}\label{solve}

Let us start with a definition.

\begin{defi}[$c$-concave function and $c$-subdifferential]\label{cost} Let $\phi : \S \longrightarrow \mathbb{R}\cup \{- \infty\}$ be a function. We define $\phi^c : \S \longrightarrow \mathbb{R}\cup \{-\infty,\, +\infty\}$,  the $c$-transform of $\phi$ by the formula 
$$  \phi^c(x) = \inf_{n \in \S} c(n,x)-\phi(n).$$
Such a function $\phi$ is said to be $c$-concave if for all $x \in \S$, $\phi^c(x) < + \infty$ (so that $(\phi^c)^c$ is well-defined) and $(\phi^c)^c=\phi$
(in the rest of the paper, we write $\phi^{cc}$). The $c$-subdifferential of a function $\phi$ is defined by the formula 
$$\partial_c \phi = \{(n,x) \in \S\times \S; \phi(n) + \phi^c(x) = c(n,x) \}.$$
The $c$-subdifferential at $n$ is defined as the subset of $\S$: 
$$ \partial_c\phi (n) = \{ x \in \S ;  \phi(n) + \phi^c(x) = c(n,x) \}.$$
\end{defi}

The overall idea of the proof is to build a pair of Lipschitz c-concave functions $(\phi,\phi^c)$ such that there exists an optimal plan $\Pi_0$ whose support satisfies
\begin{equation}\label{condi}
 Supp (\Pi_0) \subset  \partial_c \phi.
\end{equation}
Indeed, this yields
$$  \int_{\S} \phi(n) d\sigma(n) + \int_{\S} \phi^ c(x) d\mu(x) = \int_{\S \times \S} \phi(n)+\phi^ c(x) \, d\Pi_0(n,x)=\int_{\S \times \S} c(n,x)\, d\Pi_0(n,x)$$

and $(\phi,\phi^c)$ is then a maximizing pair. 

We will need the following regularity result on c-transform. 
\begin{prop}\label{elip} Let  $\psi : \S  \longrightarrow \R \cup \{- \infty\}$ be an upper semicontinuous function such that 
\begin{equation}\label{jesaisplus}
 \forall x \in \S \;\;\; \sigma(\{ \psi >-\infty\} \cap B(x,\pi/2)) >0.
 \end{equation}
Then, $\psi^c$ is real-valued and Lipschitz regular on $\S$. Moreover, $\psi^{cc}\geq \psi$.
\end{prop}
\begin{proof} 
First, note that $\psi^c(\S) \subset \R\cup \{-\infty\}$ follows from the definition of the $c$-transform and (\ref{jesaisplus}). Now, suppose there exists $n_0$ such that $\psi^c(n_0)=-\infty$. Combining the upper semicontinuity of $\psi$: 
$$ \liminf_{x_k \rightarrow x_{\infty}} c(n_0,x_k) - \psi (x_k) \geq c(n_0,x_{\infty})-\psi(x_{\infty} )\geq  -\psi(x_{\infty} )$$
together with the compactness of $\S$ yields the existence of $x_0$ such that
$$ -\infty= \psi^c(n_0) \geq -\psi (x_0)$$
contradicting $\psi(\S) \subset \R\cup \{-\infty\}$. Therefore $\psi^c$ is real-valued. Repeating the very same argument also gives for all $n \in \S$,
\begin{equation}\label{minlip} \psi^c(n) = \min_{x \in \S} c(n,x)- \psi(x).
\end{equation}

Second, we prove that $\psi^c: \S \rightarrow \R$ is continuous.  Let $(n_k)_{k\in \N}$ be a sequence converging to $n$. Using (\ref{minlip}) and the upper semicontinuity of $\psi$, we get
$$ \liminf_{k \rightarrow + \infty} \psi^c (n_k) \geq \psi^c(n).$$
On the other hand, since $\psi^c$ is defined as an infimum of continuous functions, it is an upper semicontinuous function. After all, $\psi^c$ is continuous and, in particular, bounded. As a consequence, $\psi^{cc}$ is well-defined and for all $n$, there exists $x$ such that 
$$ \psi^{cc}(n) = c(n,x) - \psi^c(x) \geq c(n,x) - (c(n,x)-\psi(n))= \psi(n)$$
by definition of the $c$-transform.

It remains to prove that $\psi^c$ is Lipschitz, i.e. $\sup_{n,n'} (\psi^c(n)-\psi^c(n'))/d(n,n') <+\infty$. Since $\S$ is compact and $\psi^c$ is bounded, it suffices to prove that $\psi^c$ is locally Lipschitz in the neighborhood of any point $n \in \S$. To this aim, observe that we now have for all $(n,x) \in \S \times \S$, 
$$\psi(x)+\psi^c(n) \in \R \cup \{-\infty\}.$$
Since the cost function is nonnegative, we infer
$$ \partial_c \psi= \{(n,x) \in \S\times \S; \psi(x) +\psi^c(n)=c(n,x)\} \subset \{c<+\infty\}. $$
Moreover, the continuity of $\psi^c$ and $c$, and the upper semicontinuity of $\psi$ yield that $\partial_c\psi$ is a closed (hence compact) subset of $\S\times \S$. Therefore, the definition of the cost function guarantees the existence of  $\alpha>0$ such that
 $$ \partial_c \psi \subset \{(n,x)  \in \S\times \S; \, d(n,x)<{\pi}/{2}- \alpha \}.$$
 Recall that $c$ is Lipschitz on  $\{(n,x)  \in \S\times \S; \, d(n,x)<{\pi}/{2}- \alpha \}$; we set $L_{\alpha}$ its Lipschitz constant. According to (\ref{minlip}), for all $n \in \S$, there exists $x_n$ such that $(n,x_n) \in \partial_c \psi$. By definition of $\psi^c$, we have for any $z$ close to $n$ (say $z \in B(n,\alpha/2)$),
$$ \psi^c(z) -\psi^c(n) \leq (c(z,x_n)- \psi(x_n)) -(c(n,x_n) -\psi(x_n))= c(z,x_n)- c(n,x_n) \leq L_{\alpha/2} \,d(z,n).$$
 Using the same argument focusing on $z$ instead of $n$, we get the same estimate for $\psi^c(n) -\psi^c(z) $. Therefore $\psi^c$ is locally Lipschitz on $ B(n,\alpha/2)$ and the proof is complete.
  \end{proof}

 To complete the proof of this part, it remains to build a $c$-concave map satisfying (\ref{condi}). To this aim, we need to recall the notion of c-cyclically monotonicity introduced by Knott and Smith \cite{KS}.

 \begin{defi}[$c$-cyclically monotone set] A subset $S \subset \S \times \S$ is called a $c$-cyclically monotone set if for any  integer $s>0$ and any pairs $(n_1,x_1), \cdots ,(n_s,x_s) \, \in S$, the following inequality is satisfied:
$$ c(n_2,x_1) + c(n_3,x_2) +\cdots +c(n_1,x_{s}) \geq c(n_1,x_1) + \cdots +c(n_s,x_s) .
$$
\end{defi}
\begin{remark} When the underlying marginals are finite combination of Dirac masses, optimality of a transport plan is equivalent to $c$-cyclically monotonicity  of its support. By approximation, it can be proved that whatever the marginals are,  the support of an optimal plan is a $c$-cyclically monotone set when the cost function $c$ is a continuous map \cite[Theorem 2.3]{GMcC}.
\end{remark} 
 
The $c$-subdifferential of a $c$-convave function is a $c$-cyclically monotone set. The main result of this section is the following

\begin{thm}\label{main}Let $\mu$ be a Borel probability measure on $\S$ satisfying  (\ref{assump}) and $\Gamma_0(\sigma,\mu)$ be the set of optimal plans in the transport problem (\ref{TP}). There exists a  Lipschitz $c$-concave map $\phi$ such that
$$\Gamma:= \bar{\bigcup_{\Pi_0 \in \Gamma_0(\sigma,\mu)} Supp (\Pi_0) }\bigcap \{ c < + \infty\} \subset \partial_c \phi.$$
\end{thm}

\begin{remark} This result is a generalization of a construction due to Rockafellar which gives a convex function from a cyclically monotone set (defined similarly in terms of the standard scalar product instead of $c$). When the cost function is a lower semi-continuous real-valued map, Theorem \ref{main} is known as the Rockafellar-Rüschendorff Theorem \cite{ruschen96} (the result holds true for {\it any} $c$-monotone set in this setting). Our proof will be along the same lines; however due to the fact that the cost function assume infinite values, there are additional difficulties to prove the map $\phi$ is well-defined.
\end{remark}

\begin{proof}
 We first prove that  $\Gamma$ is a $c$-cyclically monotone set. Note that being the mass transport problem linear, any finite convex combination of optimal plans is an optimal plan. Thus, any  $(n_1,x_1), \cdots ,(n_s,x_s) \, \in \bigcup_{\Pi_0 \in \Gamma_0(\sigma,\mu)} Supp (\Pi_0) $ belong to the support of an optimal plan. As recalled above, the support of any optimal plan relative to $c$ is $c$-cyclically monotone.  As a consequence, the set 
 $\bigcup_{\Pi_0 \in \Gamma_0(\sigma,\mu)} Supp (\Pi_0) $ is $c$-cyclically monotone.  Furthermore, a simple approximation argument based on the continuity of  the cost function  shows that the closure of any $c$-cyclically monotone set is $c$-cyclically monotone. $\Gamma$ is then $c$-cyclically monotone as a subset of a set satisfying this property.

 According to the result in Section \ref{WP}, 
 $$ \int c \, d\Pi_0 < +\infty$$
 whenever $\Pi_0 \in \Gamma_0(\sigma,\mu)$. Therefore, 
 $\Pi_0(\{c<+\infty\})=1$
 and $\Gamma$ is a set of full $\Pi_0$-measure.

We fix $(n_0,x_0) \in \Gamma$ and define
$$ \varphi (n) = \inf_{s \in \N} \, \inf \big\{ \sum_{i=0}^s c(n_{i+1},x_{i})  -c(n_i,x_i) \,;\, \forall i \in \{1, \cdots, s\} \,(n_i,x_i) \in \Gamma \big\}$$
where $n_{s+1}=n$.
By assumption, $\Gamma \subset \{ c < + \infty \}$ thus, for  $(n_1,x_1) \cdots (n_s,x_s) \in \Gamma$, the term inside the brackets above is in $\R \cup \{+ \infty\}$ and the map $\varphi$ is well-defined. Moreover $\varphi (\S) \subset \R \cup \{\pm \infty \}$.

The first step is to prove that $\varphi(n) < + \infty$ for all $n$. We start with proving $\varphi(n_0) = 0$.
 Taking $s=0$ in the definition above leads to the inequality 
 \begin{equation}\label{s=0}
 \varphi(n) \leq c(n,x_0)-c(n_0,x_0).
 \end{equation}
 In particular, $\varphi(n_0) \leq 0$. Let  $(n_1,x_1), \cdots ,(n_s,x_s) \, \in \Gamma$. The term appearing in the definition of $\varphi(n)$ when $n=n_0$ reads
 $$  c(n_1,x_0) + \cdots + c(n_0,x_s) - (c(n_0,x_0) + \cdots c(n_s,x_s)) \geq 0$$
 since $\Gamma$ is $c$-cyclically monotone. After all $\varphi(n_0) = 0$.

  As a particular case, consider $n \in B(x_0, \pi/2)$. By definition of the cost function and using (\ref{s=0}), we get $\varphi(n)<+ \infty$ for such a $n$. For general $n$, we will show that there exists a finite chain of points from $n_0$ to  $n$ such that, roughly speaking, the previous condition is satisfied. More precisely, we will prove that there exists $k \in \N\setminus\{0\}$ and $(n_i,x_i)_{1\leq i \leq k} \in \Gamma^k$ such that
$$ \begin{array}{ll}
 & c(n_1,x_0) < + \infty \\
 \forall i \in \{1,\cdots,k-1\} & c(n_{i+1},x_i) < + \infty  \\
 & c(n,x_k) < + \infty. 
 \end{array}
 $$
To this aim, let us define by induction the following {\it nondecreasing} sequence of measurable sets (as analytic sets):
$$
\begin{array}{l}
A_0 = \{x_0\} \\
A_{i+1} = p_x(p_n^{-1} ((A_i)_{\pi/2})\cap \Gamma)
\end{array}
$$
where $p_n$ and $p_x$ stand for the projections on the $n$ and $x$ coordinates respectively. Our goal is to show that for sufficiently large $k$, $(A_k)_{\pi/2}= \S$ which, by definition of $c$, implies the existence of a chain described above of length $k$. Indeed, $n_{i+1} \in (A_{i})_{\pi/2}$ means there exists $x_{i} \in A_{i}$  at distance less than $\pi/2$ from $n_{i+1}$. Moreover, by definition of $A_{i}$, there exists $n_i$ such that  $(n_i,x_i) \in \Gamma$ and $n_i \in (A_{i-1})_{\pi/2}$. The existence of the finite chain follows by an inductive argument.

Choose any optimal transport plan $\Pi_0$ in $\Gamma_0(\sigma,\mu)$. We estimate the masses as follows
\begin{align*}
 \mu (A_{i+1})  \underset{\Pi_0 \mbox{ is a plan } (\ref{planmargi})}{=}\Pi_0(p_x^{-1}(A_{i+1}))\underset{\mbox{def. of } A_{i+1}}{\geq }&   
 \Pi_0(p_n^{-1} ((A_i)_{\pi/2})\cap \Gamma) \\ 
 \underset{\Pi_0(\Gamma)=1\;\;\;\;\;\;\;\;}{= }& \Pi_0(p_n^{-1} ((A_i)_{\pi/2})) \underset{\Pi_0 \mbox{ is a plan } (\ref{planmargi})}{=}\sigma ((A_i)_{\pi/2}).
 \end{align*}
Note that $ \S \setminus (A_i)_{\pi/2} = A_i^*$ is a convex set. Since $A_i\supset A_0\neq \emptyset$, $A_i^*\neq \S$ but $A_i^*$ is empty when $(A_i)_{\pi/2}=\S$. Therefore, either $(A_i)_{\pi/2}= \S$ or (\ref{rein}) (which is equivalent to $\sigma(\S \setminus \omega^*) > \mu (\omega) + \epsilon$) gives us 
$$ \mu (A_{i+1})\geq  \sigma ((A_i)_{\pi/2}) = \sigma(\S \setminus A_i^*) \underset{(\ref{convstar})}{=} \sigma(\S \setminus (conv(A_i))^*) \geq \mu (conv(A_i)) + \epsilon \geq \mu (A_i) + \epsilon. $$
(Note that $Conv(A_i)=\S$ would imply $A_i^*=\emptyset \Leftrightarrow (A_i)_{\pi/2}= \S$.) Consequently, since $\mu$ is a probability measure,  there exists an integer $k$ such that $(A_k)_{\pi/2}= \S$. This proves the first claim.

 Using the same approach, we will show that $\varphi > -\infty$ on $p_n(\Gamma)$ hence $\sigma-$almost everywhere. Once again, consider as a particular case a point $n \in p_n(\Gamma)$ for which there exists $x$ such that $(n,x) \in \Gamma$ and $c(n_0,x) < +  \infty$. Take $n' \in \S$ and $(n_i,x_i)_{1\leq i\leq s} \in \Gamma^s$. Using $(n_i,x_i)_{1\leq i\leq s+1}$ where $(n_{s+1},x_{s+1})= (n,x)$, we get by definition of $\varphi$
 $$ \varphi(n') \leq \sum_{i=1}^s (c(n_{i+1},x_i) -c(n_i,x_i))   + c(n',x)-c(n,x).$$
 Since $(n_i,x_i)_{1\leq i\leq s} \in \Gamma^s$ is arbitrary, we infer for all $n' \in \S$,
\begin{equation}\label{ccon}
\varphi(n') \leq \varphi(n) + c(n',x) -c(n,x).
\end{equation}
In particular, the formula above with $n'= n_0$ together with $\varphi(n_0)=0$ gives $\varphi(n) > -\infty$.

By an inductive argument, we generalize (\ref{ccon}) to
\begin{equation}\label{trick13}
 \varphi(n') \leq \varphi (n) + \big(c(\tilde{n}_1,x) - c(n,x)\big)+  \cdots +\big(c(\tilde{n}_{k},\tilde{x}_{k-1})- c(\tilde{n}_{k-1},\tilde{x}_{k-1})\big)+ \big(c(n',\tilde{x}_k)-c(\tilde{n}_k,\tilde{x}_k)\big)
\end{equation}
where $(\tilde{n}_i,\tilde{x}_i) \in \Gamma$ for all $i \in \{1,\cdots,k\}$. As a consequence, (\ref{trick13}) with $n'=n_0$ allows us to conclude provided we can find $(\tilde{n}_i,\tilde{x}_i)_{1\leq i\leq k} \in \Gamma^k$ such that 

$$ \begin{array}{ll}
 & c(\tilde{n}_1,x) < + \infty \\
 \forall i \in \{1,\cdots,k-1\} & c(\tilde{n}_{i+1},\tilde{x}_i) < + \infty  \\
 & c(n_0,\tilde{x}_k) < + \infty. 
 \end{array}
 $$
In other terms, we are done if there exists a finite chain from $n$ to $n_0$. But the previous arguments apply {\it verbatim} with $A_0=\{x\}$ and $(n,x)\in \Gamma$. Consequently, $\varphi>-\infty$ on $p_n(\Gamma)$. In particular, since $\sigma(p_n(\Gamma))=\Pi_0(\Gamma)=1$, 
$$\sigma(\{\varphi=-\infty\})=0.$$

Now, we prove that $\Gamma \subset \partial_c \varphi$. To this aim, fix $(n,x) \in \Gamma$. Using (\ref{ccon}) we get, since $n'$ is arbitrary,
$$ c(n,x)-\varphi(n) \leq \inf_{n' \in \S} c(n',x) -\varphi(n'). $$
The reverse inequality being straightforward, this implies equality: 
\begin{equation}\label{gambl}
 c(n,x) = \varphi(n) +\varphi^c(x),
 \end{equation}
  i.e. $(n,x) \in \partial_c\varphi$.

To conclude the proof, we build out of $\varphi$, a $c$-concave function. By what precedes, $\varphi$ satisfies the assumptions of Proposition \ref{elip} (recall that by definition, $\varphi$ is an infimum of continuous functions). We get by applying this proposition to $\varphi$ and then to $\varphi^c$ that $\varphi^c$ and $\varphi^{cc}$  are (real-valued) Lipschitz functions and 
\begin{equation}\label{blurb}
\varphi^{cc} \geq \varphi.
\end{equation}
 By definition of the $c$-transform, we infer from (\ref{blurb}) the inequality $\varphi^{ccc} \leq \varphi^c$.  On the other hand, applying Proposition \ref{elip} to $\varphi^c$ leads to
 $ \varphi^{ccc} \geq \varphi^c$. Consequently, $ \varphi^{ccc} = \varphi^c$ and $\phi:= \varphi^{cc}$ is a Lipschitz $c$-concave map. To complete the proof, it remains to check that $\Gamma \subset \partial_c \phi$. Take $(n,x) \in \Gamma$ and recall (\ref{gambl}): $\varphi (n)+\varphi^c (x)= c(n,x)$. This yields
 $$\varphi(n) = c(n,x)-\varphi^c(x) \geq \varphi^{cc}(n)$$
 which combined with (\ref{blurb})  gives
 $$ \phi= \varphi^{cc}= \varphi \mbox{ on } p_n (\Gamma).$$
 Since $\phi^c =\varphi^{ccc}= \varphi^c$ as proved above, (\ref{gambl}) is actually equivalent to $\phi(n)+\phi^c(n)=c(n,x)$ and the proof is complete.  
 \end{proof}

\subsection{Proof of Theorem \ref{transp}: the uniqueness part}\label{recover}
In Section \ref{solve}, we proved the existence of a solution of Kantorovitch's variational problem $(\phi,\phi^c)$ where $\phi$ is a Lipschitz $c$-concave map. It remains to prove the uniqueness of this pair. We start with the following observation, called the double complexification trick. Let $(\phi,\psi) \in {\mathcal A}$. By definition of ${\mathcal A}$, we have
$$ J(\phi,\psi) \leq J(\phi, \phi^c) \leq J(\phi^{cc},\phi^c)$$
where $J(\phi,\psi) = \int_{\S} \phi(n) d\sigma(n) + \int_{\S} \psi(x) d\mu(x)$ (the $c$-transform of a Lipschitz function is a Lipschitz function as well, thanks to Proposition \ref{elip}). So, any maximising pair $(\phi,\psi)$ has to satisfy
$$
\begin{array}{cc}
\psi= \phi^c & \mu\mbox{-a.e.}\\
\phi=\phi^{cc} & \sigma\mbox{-a.e.}
\end{array}
$$
By continuity, this implies  $\phi=\phi^{cc}$, and $\phi$ and $\phi^c$ are Lipschitz $c$-concave maps. (Similarly, $ J(\phi,\psi) \leq J(\psi^c, \psi)$ yields $\psi^c=\phi$.) Therefore, the set $\partial_c\phi=\{(n,x) \in \S \times \S; \phi(n)+\phi^c(x)=c(n,x)\}$ is contained in the open set $\{c<+\infty\}$ and is a compact subset of $\S\times \S$.

Now, for any optimal plan $\Pi \in \Gamma_0(\sigma,\mu)$, we have
$$  \int_{\S} \phi(n) d\sigma(n) + \int_{\S} \phi^ c(x) d\mu(x) =\int_{\S} \phi(n) d\sigma(n) + \int_{\S} \psi(x) d\mu(x) = \int_{\S \times \S} c(n,x)\, d\Pi(n,x)$$
Subtracting the left-hand term to the right-hand term yields $ c(n,x) -\phi(n)-\phi^c(x) = 0 \;\Pi-\mbox{a.e.}$. By compactness of $ \partial_c \phi$, this yields
$$ \supp (\Pi) \subset \partial_c \phi.$$
The above property holds true for every optimal plan. Therefore, using the compactness of $ \partial_c \phi$ again, we obtain 
\begin{equation}\label{me41}
\bar{\bigcup_{\Pi_0 \in \Gamma_0(\sigma,\mu)} Supp (\Pi_0) }\bigcap \{ c < + \infty\} \subset \partial_c \phi\subset \{ c < + \infty\}
\end{equation}
namely the condition that appears in Theorem \ref{main}. It remains to prove that this condition determines $\phi$ up to adding a constant.

We set $N=\{n \in \S; \phi \mbox{ is not differentiable at } n\}$. Since $\phi$ is Lipschitz, $\sigma(N)=0$. Let $n$ be in $\S\setminus N$ and $x \in \partial_c \phi (n)\subset B(n,\pi/2)$. By definition of $\phi^c$, this implies that the map $z \longmapsto c(z,x)-\phi(z)$ admits a minimum at $n$, therefore its gradient equals ${0}$. This gives (with a slight abuse of notation)
$$ c'(d(n,x)) \nabla d_x(n) = \nabla \phi (n)$$
(where $d_x(n):=d(n,x)$ and $c'(0)=0$) which entails
$$ \arctan ( (|\nabla \phi (n)|) = d(n,x) < \frac{\pi}{2}.$$
Using $x = \exp_n(-d(n,x)\nabla d_x(n))$, we get that $x$ is unique and given by the expression
$$ T(n) = \exp_n\Big(-\frac{\arctan (|\nabla \phi (n)|)}{|\nabla \phi (n)|}\nabla \phi (n)\Big).$$
We have proved that
$$ \Gamma \cap (\S\setminus N \times \S) \subset \{(n,T(n); n \in \S\setminus N\}.$$
Therefore, if two Lipschitz $c$-concave maps $\phi_1$ and $\phi_2$ satisfy (\ref{me41}) then for $\sigma$-almost every $n$, $\nabla \phi_1(n)=\nabla \phi_2(n)$. To conclude, we use that a Lipschitz function whose derivative is null $\sigma$-a.e. is a constant function. 

\begin{remark}\label{rem} We actually proved that under the assumptions of the theorem above, there exists a unique optimal transport plan and this plan is induced by the map $T$. Once the regularity of $\phi$ is proved, the strategy employed in the proof above is now classical. On Riemannian manifolds, it is due to Mc Cann \cite{McC01}. Another proof in this particular case is given in \cite{Oliker}.

The proof of Theorem \ref{BMCC} goes as follows. According to Remark \ref{rem42}, we get that for all $\omega \in \C\setminus \{\S\}$
$$ \mu(\S \setminus \omega) > (f\sigma) (\omega^ *). $$
Note that the arguments required to prove Proposition \ref{prop42} remain true if $\sigma$ is replaced by $(f\sigma)$. As a consequence,
the reinforced assumption (\ref{rein}) holds with $f\sigma$ in place of $\sigma$. With this property at our disposal, it suffices to repeat the arguments in Sections 4.2 and 4.3 with $f\sigma$ instead of $\sigma$ to get a proof of Theorem \ref{BMCC}.
\end{remark}


\section*{Appendix}

In this appendix, we prove the following result 

\begin{lemma}\label{parti} Let $\mu$ be a finite Borel measure on the unit sphere $\Sp^m$ endowed with its canonical distance $d$. Suppose that $\mu$ is absolutely continuous with respect to the standard uniform measure on $\Sp^m$ and $\mu(\Sp^m)$ is a positive rational number. Then for any $\alpha>0$, there exists a finite partition $(P_i)_{1\leq i\leq K}$ of $\S$ (depending on $\alpha$) such that for all $i$, $\mu(P_i)>0$ is a rational number and $diam (P_i) < \alpha$.
\end{lemma}

\begin{remark} When $\mu$ is the uniform probability measure $\sigma$, the proof below together with the expression of $\sigma$ in polar coordinates guarantee that we can further require  $\forall i, \;\sigma (P_i)= 1/M$, being $M$ a sufficiently large integer.  
\end{remark}
\begin{proof}
The proof is by induction on the dimension $m$. For $m=1$, fix a number $\alpha_1>0$. Then, partition $\Sp^1$ into finitely many left-open, right-closed segments $(I_j)_{1\leq j\leq K_1}$ whose length $l(I_j)$ satisfies $ l(I_j)<\alpha_1$ and $\mu(I_j) \in \Q$ (we use that $s \mapsto \mu((a,s])$ is continuous); $\mu(I_{K_1}) \in \Q$ is guaranteed by $\mu(\Sp^1)\in \Q$. For $m=2$, fix a point $N\in \Sp^2$ and $\alpha_2>0$. Consider a partition $(C_i)_{1 \leq i \leq K_2}$ where $C_1$ is the closed ball with radius $R_1$ and center $N$, $C_i= \{z\in \Sp^2; R_{i-1} <  d(N,z) \leq R_i\}$ for $ i \in \{2,\cdots,K_2-1\}$ and $C_{K_2}$ is the closed ball with radius $\pi-R_{K_2}$ and center $-N$. We require that the $(R_i)$ satisfy: 
$$ \alpha_2/2 < R_1< \alpha_2, \quad    \alpha_2/2 < R_i-R_{i-1} < \alpha_2, \quad  \pi-R_{K_2} < \alpha_2$$
and $\mu(C_i)\in \Q$. Let us set $p$ the projection onto the equator relative to $N$. Note that the measures $(p_{\sharp} (\mu \res C_i))_{1 \leq i\leq K_2})$ are absolutely continuous with respect to the uniform measure on the circle thus, applying the case $m=1$ to all the measures $(p_{\sharp} (\mu \res C_i))_{1 \leq i\leq K_2})$, we get a partition $(P_i)_{1\leq i\leq K}$ of $\Sp^2$ (namely $(C_i \cap p^{-1}(I_j^i))_{i,j}$, being $(I_j^i)_j$ the partition corresponding to $p_{\sharp} (\mu \res C_i)$) such that $\mu(P_i) \in \Q$. Moreover, the expression of the spherical distance in polar coordinates implies that the diameter of any $P_i$ is smaller than $\alpha$ provided $\alpha_1$ and $\alpha_2$ are sufficiently small. The higher dimensional case easily follows from the arguments used for $m=2$.

\end{proof}


\bibliographystyle{amsplain}
\providecommand{\bysame}{\leavevmode\hbox to3em{\hrulefill}\thinspace}
\providecommand{\MR}{\relax\ifhmode\unskip\space\fi MR }
\providecommand{\MRhref}[2]{%
  \href{http://www.ams.org/mathscinet-getitem?mr=#1}{#2}
}
\providecommand{\href}[2]{#2}

\end{document}